\documentclass[reqno]{amsart}
\usepackage{amsmath,amsfonts,amssymb}
\usepackage{amsthm}
\usepackage{ifthen}
\usepackage{longtable}
\usepackage{array}
\usepackage{url}

\newcommand{\Z}{\mathbb{Z}}
\newcommand{\Q}{\mathbb{Q}}

\newcommand{\C}{\mathbb{C}}
\newcommand{\N}{\mathbb{N}}
\newcommand{\FF}{\mathbb{F}}
\newcommand{\ord}{\mathfrak{o}}

\newcommand{\p}{\mathfrak p}

\newcommand{\bbeta}{\boldsymbol \beta}
\newcommand{\Sfin}{S_{\mathrm{fin}}}
\newcommand{\Sinf}{S_{\infty}}

\DeclareMathOperator{\Gal}{Gal}

\newtheorem{definition}{Definition}
\newtheorem*{theorem*}{Theorem}
\newtheorem{theorem}{Theorem}
\newtheorem{lemma}{Lemma}

\newtheorem{proposition}{Proposition}

\newtheorem{claim}{Claim}

\theoremstyle{remark}
\newtheorem{remark}{Remark}

\begin{document}
\title[On corner avoidance of $\boldsymbol \beta$-adic Halton sequences]{On corner avoidance of $\boldsymbol \beta$-adic Halton sequences}

\author[M. Hofer]{Markus Hofer}
\address{M. Hofer \newline
\indent Graz University of Technology, \newline
\indent Institute of Mathematics A,\newline
\indent  Steyrergasse 30,\newline
\indent  8010 Graz, Austria.}
\email{markus.hofer\char'100tugraz.at}

\author[V. Ziegler]{Volker Ziegler}
\address{V. Ziegler\newline
\indent Johann Radon Institute for Computational and Applied Mathematics (RICAM)\newline
\indent Austrian Academy of Sciences\newline
\indent Altenbergerstr. 69\newline
\indent A-4040 Linz, Austria}
\email{volker.ziegler\char'100ricam.oeaw.ac.at}

\begin{abstract}
 We consider the corner avoiding property of $s$-dimensional $\boldsymbol \beta$-adic Halton sequences. After extending this class of point sequences in an
intuitive way, we show that the hyperbolic distance between each element of the sequence and the closest corner of $[0,1)^s$ is
$\mathcal{O}\left(\frac{1}{N^{s/2+\varepsilon}}\right)$, where $N$ denotes the index of the element. In our proof we use tools from Diophantine analysis, more
precisely, we apply Schmidt's Subspace Theorem.
\end{abstract}

\maketitle

\section{Introduction}

In this article we consider special distributional properties of deterministic point sequences $(\mathbf{x}_n)_{n > 0}$ in the $s$-dimensional unit cube. First,
we define an $s$-dimensional interval $[\mathbf{a}, \mathbf{b} ) \subseteq [0,1)^s$ as $[\mathbf{a}, \mathbf{b}) = \{ \mathbf{x} \in [0,1)^s \colon a_i \leq x_i
< b_i, i = 1, \ldots, s\}$, where $\mathbf{x} = (x_1, \ldots, x_s)$. Furthermore, we call a point sequence $(\mathbf{x}_n)_{n > 0}$ uniformly distributed modulo
1 (u.d.), if
\begin{equation*}
 \lim_{N \rightarrow \infty} \frac{1}{N}\sum_{n = 1}^N \mathbf{1}_{[\mathbf{a}, \mathbf{b})} (\mathbf{x}_n) = \lambda_s([\mathbf{a},\mathbf{b}))
\end{equation*}
for all $s$-dimensional intervals $[\mathbf{a}, \mathbf{b} ) \subseteq [0,1)^s$, where $\lambda_s$ denotes the $s$-dimensional Lebesgue measure. An equivalent
characterization is given by a famous theorem of Weyl stating that a point sequence $(\mathbf{x}_n)_{n > 0}$ is u.d.\ if and only if for every real-valued
continuous function $f$ the relation
\begin{equation*}
  \lim_{N \rightarrow \infty} \frac{1}{N} \sum_{n = 1}^N f(\mathbf{x}_n) = \int_{[0,1)^s} f(\mathbf{x}) d\mathbf{x}
\end{equation*}
holds. This result gives a hint how such sequences can be used for numerical integration, which is usually called Quasi-Monte Carlo (QMC) integration. However,
it does not reveal anything about the size of the integration error. Fortunately, the Koksma-Hlawka inequality, see \cite{hlawka}, states that
\begin{equation*}
 \left| \frac{1}{N} \sum_{n = 1}^N f(\mathbf{x}_n) - \int_{[0,1]^s} f(\mathbf{x}) d\mathbf{x} \right| \leq V(f) D_N^*(\mathbf{x}_n),
\end{equation*}
where the star-discrepancy $D_N^*$ is defined as
\begin{equation*}
 D_N^*(\mathbf{x}_n) = D_N^*(\mathbf{x}_1, \ldots, \mathbf{x}_N) = \sup_{\mathbf{a} \in [0,1)^s} \left| \frac{1}{N} \sum_{n=1}^N
\mathbf{1}_{[\mathbf{0},\mathbf{a})} (\mathbf{x}_n) - \lambda_s([\mathbf{0},\mathbf{a}))\right|
\end{equation*}
and $V(f)$ denotes the variation of $f$ in the sense of Hardy and Krause. By the fact that the best known QMC sequences, so-called low-discrepancy sequences,
have an asymptotic star-discrepancy of order $\mathcal{O}\left(\frac{\log(N)^s}{N} \right)$, we get that the error of QMC integration converges for every $s$
faster than the corresponding error of ordinary Monte Carlo integration, where random instead of deterministic point sequences are used. For a detailed
discussion of uniformly distributed sequences and discrepancy, see e.g.\ \cite{dt}.

Obviously, the Koksma-Hlawka inequality is useful only if $V(f) < \infty$, which further implies that $f$ has no singularities in $[0,1)^s$. Unfortunately,
there are many potential applications of QMC integration, where the integrand function has a singularity, for instance the pricing of Asian options in the
Black-Scholes model. For a detailed discussion we refer to \cite{owen}. However, in such cases an analogon of the Koksma-Hlawka inequality holds if the points 
of the
sequence $(\mathbf{x}_n)_{n > 0}$ stay sufficiently far away from the singularities and the integrand function satisfies mild regularity conditions. More
precisely, we have the following theorem due to Owen~\cite{owen}:

\begin{theorem*}
 Let $f(\mathbf x)$ be a real valued function on $[0,1)^s$ which satisfies
\begin{equation*}
 |\partial^u f(\mathbf{x})| \leq B \prod_{i = 1}^s \left(x^{(i)}\right)^{-A_i - \mathbf{1}_{\{i \in u\}}} \quad \text{ for all } u \subseteq \{1, \ldots, s\},
\end{equation*}
where $\mathbf{x} = (x^{(1)}, \ldots, x^{(s)})$, $A_i > 0$ and $B < \infty$. Suppose that $\left(\mathbf{x}_n\right)_{n > 0}$ satisfies
\begin{equation}\label{caprop}
\prod_{i = 1}^s x_n^{(i)} \geq cN^{-r}
\end{equation} 
for all $1\leq n \leq N$. Then for any $\eta > 0$
\begin{equation*}
 \left| \frac{1}{N} \sum_{n = 1}^N f(\mathbf{x}_n) - \int_{[0,1]^s} f(\mathbf{x}) d\mathbf{x} \right| \leq C_1 D_N^*(\mathbf{x}_n) N^{\eta + r \max_i A_i} + C_2
N^{r (\max_i A_i - 1)}
\end{equation*}
holds for constants $C_1$ and $C_2$ which may depend on $\eta$.
\end{theorem*}

For our purposes it is sufficient to consider only singularities in the corners of the $s$-dimensional unit cube, since if there is only a finite number of
singularities, every integrand function can be transformed appropriately. In the sequel, we say $(\mathbf{x}_n)_{n > 0}$ is (hyperbolically) avoiding the 
origin,
if \eqref{caprop} holds. Owen \cite{owen} presents similar findings in case that singularities in the origin and in the corner $\mathbf{1} = (1, \ldots, 1)$ 
have to be avoided simultaneously. These results have been extended by Hartinger et al.\ \cite{hkz} who give bounds for the integration error, when there are 
singularities in all corners of the unit cube and the underlying point sequences avoid all corners hyperbolically.

The corner avoidance property has been investigated for many low-discrepancy sequences. First results are due to Sobol \cite{sobol} and Owen \cite{owen}, who
investigate the origin avoiding property of Sobol and Halton sequences. In \cite{hkz}, Hartinger et al.\ investigate Halton, Faure and generalized Niederreiter
sequences and their hyperbolic distance to all corners of the unit hypercube, where they compute the optimal values for $r$ in \eqref{caprop}. Hartinger and
Ziegler \cite{hz} extend these results to random start Halton-sequences.

In the present article we consider the $\boldsymbol \beta$-adic Halton sequence, introduced in \cite{hit}, which is a natural extension of Halton's original
construction. More precisely, we consider numeration systems $G = (G_n)_{n \geq 0}$, with base sequence $G_n$ constructed by a linear recurrence of length $d$,
i.e.
$$
 G_{n + d} = a_0 G_{n + d -1} + \ldots + a_{d - 1} G_n,
$$
where $d \geq 1$, $G_0=1$, $G_k=a_0 G_{k-1}+\dots +a_{k-1} G_0+1$ for $k<d$ and $a_i \in \mathbb{N}_0$ for $i = 0, \ldots, d-1$. Then every positive integer $n$
can be represented as
\begin{equation*}
  n=\sum_{k=0}^{\infty}\varepsilon_k(n) G_k,
\end{equation*}
where $\varepsilon_k(n) \in \{0, \ldots, \lfloor G_{k+1}/G_k \rfloor - 1 \}$ and $\lfloor x \rfloor$ denotes the integer part of $x \in \mathbb{R}$. This 
expansion (called $G$-expansion) is uniquely determined provided that 
\begin{equation}\label{regular}
\sum_{k=0}^{K-1}\varepsilon_k(n) G_k < G_K,
\end{equation}
where the digits $\varepsilon_k(n)$ are computed by the greedy algorithm (see for instance \cite{Fraenkel}). We call a sequence of digits $(\varepsilon_0,
\varepsilon_{1}, \ldots)$ regular if it satisfies \eqref{regular} for every $K$. From now on we assume that all $G$-expansions are regular.

Let
\begin{equation}\label{alpha}
 X^{d} - a_0 X^{d - 1} - \ldots - a_{d - 1}
\end{equation}
be the characteristic polynomial of the numeration system $G=(G_n)_{n\geq 0}$. In this article we are interested only in numeration systems, where the 
characteristic 
polynomial is irreducible and the dominant root $\beta$ of the characteristic polynomial \eqref{alpha} is a Pisot number. In this case we call $\beta$, which 
plays a crucial role, the characteristic root of the numeration system $G$. Note that, by \cite[Theorem 2]{frougny}, the dominant root $\beta$ of 
\eqref{alpha} is always a Pisot number if
$$
a_0 \geq \ldots \geq a_{d-1} \geq 1.
$$
In the sequel we write $\mathbf{a} = (a_0, \ldots, a_{d-1})$.

Let $\mathcal K_G$ be the set of all sequences $(\varepsilon_j)_{j\geq 0}$ such that 
\begin{equation*}
\sum_{k=0}^{K-1}\varepsilon_k G_k < G_K,
\end{equation*}
for all $K>0$. We identify an integer $n$ with its regular expansion, i.e. the sequence $(\varepsilon_0(n),\varepsilon_1(n),\dots)\in \mathcal K_G$.
Now let us define
the so-called $\beta$-adic Monna map $\phi_\beta \colon \mathcal K_G \rightarrow \mathbb{R}$ given by
\begin{equation*}
\phi_\beta((\varepsilon_j)_{j\geq 0}) = \sum_{j \geq 0} \varepsilon_j(n) \beta^{-j-1},
\end{equation*}
where $\beta$ is the characteristic root of $G$. We call the sequence $(\phi_\beta(n))_{n > 0}$ the $\beta$-adic van der Corput sequence.
Note that if $d = 1$, then $(\phi_{\beta}(n))_{n > 0} = (\phi_{a_0}(n))_{n > 0}$ is the classical van der Corput sequence in base $a_0$.

In \cite[Lemma 1]{hit}, the authors prove that $\phi_\beta(\mathbb{N})$ is a dense subset of $[0,1)$ if and only if the coefficients of $G$ are of 
the form
\begin{align*}
 \mathbf{a} &= (a_0, \ldots, a_0)\quad  \text{ or }\\
 \mathbf{a} &= (a_0, a_0 - 1, \ldots, a_0 - 1, a_0),
\end{align*}
for $a_0 > 0$. Furthermore, they show that the $\beta$-adic Monna map transports the Haar measure on $G$ to the Lebesgue measure on the unit interval. These
results imply that there is an isometry between the odometer on the numeration system $G$ and the dynamical system defined by a transformation $T_{\beta}
\colon [0,1) \rightarrow [0,1)$, that satisfies $T_\beta\phi_\beta(n)=\phi_\beta(n+1)$. In order to define $T_\beta$
we follow \cite{hit}. Let us introduce
\begin{equation*}
 \mathcal K_G^0= \left\{x\in \mathcal K: \exists M_x, \forall j\geq M_x, \sum_{k=0}^j \varepsilon_kG_k<G_{j+1}-1 \right\}.
\end{equation*}
Let $x\in \mathcal K_G^0$ and put $x(j)=\sum_{k=0}^j \varepsilon_k G_k$ for some $j\geq M_x$. Then we define
$$\tau_\beta(x)=\varepsilon_0(x(j)+1)\dots \varepsilon_j(x(j)+1)\varepsilon_{j+1}(x)\varepsilon_{j+2}(x)\dots,$$
for all $x\in \mathcal K_G^0$. Let us note that this definition does not depend on the choice of $j\geq M_x$ and can easily be extended to
sequences $x\in\mathcal K_G\setminus \mathcal K_G^0$ by $\tau_\beta(x)=(0^\infty)$.
As in \cite{hit} we can define the so-called pseudo-inverse Monna map $\phi_\beta^{-1}$ by considering only $x\in \mathcal K_G^0$.
We define the transformation $T_{\beta}\colon [0,1) \rightarrow [0,1)$ by $T_\beta:=\phi_\beta \tau_\beta \phi_\beta^{-1}$.
Let us note that in view of the $\beta$-adic van der Corput sequnce we have $T_\beta \phi_\beta(n)=\phi_\beta(n+1)$ (see \cite{hit} for details).

Assuming certain number theoretical conditions, 
the $s$-dimensional $\boldsymbol \beta$-adic Halton sequence, with $\boldsymbol \beta = (\beta_1, \ldots, \beta_s)$, given as $(\phi_{\boldsymbol
\beta}(n))_{n > 0} = (\phi_{\beta_1}(n), \ldots, \phi_{\beta_s}(n))_{n > 0}$ is uniformly distributed in $[0,1)^s$, see~\cite{hit}. The connection between
ergodic theory and dynamical systems is drawn by Birkhoff's ergodic theorem, for a detailed discussion of this matter we refer to~\cite{ghl}.

Let us also note that the construction of the $\beta$-adic van der Corput sequences due to Ninomiya \cite{ninomiya} is closely related to our
construction. However in view of proving Corner avoidence properties our approach to $\beta$-adic Halton sequences seems to be favorable, since
the fact that a point $\phi_{\beta}(n)$ lies close to $0$ or $1$, is reflected by the $G$-expansion of the integer $n$ (see Lemma \ref{lem:extreme}). 

We call $Z(\varepsilon_0, \ldots, \varepsilon_{k-1})$ a cylinder set of length $k$, where $Z(\varepsilon_0, \ldots, \varepsilon_{k-1})$ is defined as the set of
all regular representations $(\varepsilon'_0, \varepsilon'_1, \ldots)\in \mathcal K_G$ for which $\varepsilon_i = \varepsilon'_i$ for $0 \leq i \leq k-1$. We will sometimes
write $Z$ for short, if the context is clear.

A further main ingredient for proving uniform distribution of the $\boldsymbol \beta$-adic Halton sequence is \cite[Theorem 5]{glt} which states that the
odometer on a linear recurring numeration system $G$ is uniquely ergodic 
and the corresponding invariant measure $\mu$ is given by
\begin{align}
 &\mu(Z) \label{mu} =\\ 
 &\frac{F_{K,0} \beta^{d-1} + (F_{K,1} - a_0 F_{K,0}) \beta^{d-2} + \ldots + (F_{K,d-1} - a_0 F_{K,d-2} - \ldots - a_{d-2} F_{K,0})}{\beta^K (\beta^{d-1} + 
\beta^{d-2} + \ldots + 1)},\notag
\end{align}
where $F_{K,r} := \# \{ n < G_{K+r}: n\in Z \}$ and $Z$ is a cylinder set of length $K$. We omit a detailed 
introduction of the ergodic properties of odometers on numeration systems and refer to \cite{glt}.

The remainder of this article is structured as follows: In the next section we generalize some results presented in \cite{hit} to numeration systems with
decreasing 
coefficients by extending the definition of the $\beta$-adic Monna map. In Section \ref{Sec:corner} we discuss the corner avoiding property for the extended
$\boldsymbol \beta$-adic Halton sequence. This property is proved by using the Subspace Theorem. Therefore we establish in Section \ref{Sec:aux} several 
auxiliary results that will enable us to prove this property for the $\boldsymbol \beta$-adic Halton sequence in the final section (Section \ref{Sec:Subspace}).

\section{Extended $\boldsymbol \beta$-adic Halton sequences}

In this section we slightly extend the definition of $\boldsymbol \beta$-adic Halton sequences given in~\cite{hit}. Using ergodic theory we prove that these
sequences are uniformly distributed in $[0,1)^s$. In the sequel, we denote by $\mathbf{a} = (a_0, \ldots, a_{d-1})$ the coefficients of the numeration system
$G$ and we 
assume that $a_0 \geq a_1 \geq \ldots \geq a_{d-1} \geq 1$.
\begin{definition}[Extended $\beta$-adic Monna map]
 Let $G$ be a numeration system with characteristic root $\beta$ and let the $\mathcal K_G$ and $\mathcal K_G^0$ be defined as above.
Then the extended $\beta$-adic Monna map $\psi_\beta \colon \mathcal K_G \rightarrow [0,1)$ is defined as
$$
 \psi_\beta(x) = \sum_{k = 0}^\infty f(x, k),
$$
where
\begin{equation}\label{eq:defpsi}
 f(x, k) = \sum_{i=0}^{\varepsilon_{k} - 1} \mu\left(Z_x^k(i)\right),
\end{equation}
$\mu$ is given in \eqref{mu} and $Z_x^k(i)$ is the cylinder set with parameters $(\varepsilon_0(x), \ldots, \varepsilon_{k-1}(x), i)$. As in \cite{hit}, we can
define a pseudo inverse of $\psi_{\beta}$, denoted by $\psi_{\beta}^{-1}$, by considering only $x\in \mathcal K_G^0$.
\end{definition}

\begin{remark}
Note that for a given integer $n$ the cylinders involved in sums of type~\eqref{eq:defpsi} are disjoint. 
\end{remark}

In order to illustrate the complexity of $f$, we give the following explicit example for $d = 3$, $a_0 > a_1 > a_2 \geq 1$ and $n\in \N$.
\begin{equation*}
 f(n, k) = \left \{ \begin{array}{cl} \frac{\varepsilon_k(n)}{\beta^{k+1}} & \text{ if } \scriptstyle{\varepsilon_k(n) < a_2},\\ 
	\frac{a_2}{\beta^{k+1}} + \frac{(\varepsilon_k(n) - a_2) (\beta^2 + \beta)}{\beta^{k+1} (\beta^2 + \beta + 1)}, & \text{ if }  
\scriptstyle{\varepsilon_k(n) < a_1 \vee (\varepsilon_k(n) \leq a_1  \wedge k = 1)}\\ 
	&\scriptstyle{\vee (\varepsilon_k(n) \leq a_1 \wedge \varepsilon_{k-1}(n) < a_2 \wedge k > 1)}, \\
	\frac{a_2}{\beta^{k+1}} + \frac{(a_1 + 1 - a_2) (\beta^2 + \beta)}{\beta^{k+1} (\beta^2 + \beta + 1)} + \frac{(\varepsilon_k(n) - a_1 - 1) 
\beta^2}{\beta^{k+1} (\beta^2 + \beta + 1)}, & \text{ if } \scriptstyle{(\varepsilon_k(n) > a_1  \wedge k = 1)}\\ 
	&\scriptstyle{\vee (\varepsilon_k(n) > a_1 \wedge \varepsilon_{k-1}(n) < a_2 \wedge k > 1)},\\
	\frac{a_2}{\beta^{k+1}} + \frac{(a_1 - a_2) (\beta^2 + \beta)}{\beta^{k+1} (\beta^2 + \beta + 1)} + \frac{(\varepsilon_k(n) - a_1) 
\beta^2}{\beta^{k+1} (\beta^2 + \beta + 1)}, & \text{ otherwise}. \end{array} \right.
\end{equation*}

\begin{lemma}\label{lem1}
 Let $G$ be a numeration system with characteristic root $\beta$. Then, $\psi_\beta(\mathbb{N})$ is a dense subset of $[0,1)$. Furthermore we 
have
\begin{equation*}
 \mu(Z) = \lambda(\psi_\beta(\mathcal K_Z)),
\end{equation*}
for all cylinder sets $Z$, where $\mu$ is given in \eqref{mu}, $\lambda$ is the one-dimensional Lebesgue measure and $\mathcal K_Z$ is the set of all
regular sequences $x\in Z$.
\end{lemma}

\begin{proof}
By definition it follows that the $\mu$-measure of the union of all cylinder sets is 1 thus $0 \leq \psi_\beta(n) < 1$ for all $n \in \mathbb{N}$. We want to 
emphasize that the intersections of the cylinder sets appearing in \eqref{eq:defpsi} are of measure $0$.
Let $x \in [0,1)$, then by the definition of $\psi_\beta$, there exists a sequence $Z_n^x$ of cylinder sets with digits $(\varepsilon_0, \ldots, 
\varepsilon_{n-1})$, such
that $x \in Z_n^x$ for every $n$. Furthermore, we define an increasing sequence of integers $k_n$ as
\begin{equation*}
 k_n = \sum_{i = 0}^{n-1} \varepsilon_i G_i.
\end{equation*}
Since $\lim_{n \rightarrow \infty}\mu(Z_n^x) = 0$, we have that $\lim_{n \rightarrow \infty} \psi_{\beta}(k_n) = x$. Finally it follows by construction that 
$\psi_\beta$ is measure preserving (see also \cite[Theorem 5]{glt}) .
\end{proof}

The following result generalizes Theorems 2 and 3 in \cite{hit}.

\begin{theorem}\label{main1}
 Let $G^{(1)}, \ldots, G^{(s)}$ be numeration systems defined by
\begin{equation*}
 G^{(i)}_{n+d_i} = a_0^{(i)} G^{(i)}_{n+d_i-1} + \ldots + a^{(i)}_{d_i-1} G^{(i)}_{n}
\end{equation*}
 where $a^{(i)}_0 \geq a^{(i)}_1 \geq \ldots \geq a^{(i)}_{d_i-1} \geq 1$ holds for $i = 1, \ldots, s$.
Furthermore let $\gcd\left(a_0^{(i)}, \ldots, a^{(i)}_{d_i - 1}, a_0^{(j)}, \ldots, a^{(j)}_{d_j - 1}\right) = 1$ for all $i \neq j$ and let 
$\frac{(\beta^{(i)})^{k}}{(\beta^{(j)})^l} \notin \mathbb{Q}$, for all $l, k \in \mathbb{N}$, where $\boldsymbol \beta = 
\left(\beta^{(1)}, \ldots, \beta^{(s)}\right)$ denotes the $s$-tuple of characteristic roots of the numeration systems $G^{(1)}, \ldots, G^{(s)}$. Then, the 
extended $s$-dimensional $\boldsymbol \beta$-adic Halton sequence 
\begin{equation*}
 \left(\psi_{\boldsymbol \beta}(n)\right)_{n > 0} = \left(\psi_{\beta^{(1)}}(n), \ldots,\psi_{\beta^{(s)}}(n) \right)_{n > 0}
\end{equation*}
 is uniformly distributed in $[0,1)^s$.
\end{theorem}

\begin{proof}
Since $a^{(i)}_0 \geq a^{(i)}_1 \geq \ldots \geq a^{(i)}_{d_i-1} \geq 1$ we know by \cite[Theorem 4.1]{solomyak} that the odometer on the base system 
$G^{(i)}$ has a purely 
discrete spectrum, which is given by
\begin{equation*}
 \Gamma_j = \{z \in \mathbb{C}: \lim_{n \rightarrow \infty} z^{G^{(j)}_n} = 1 \}.
\end{equation*}
By a standard result from ergodic theory (see e.g.\ \cite{ghl}), the Cartesian product system constructed by the odometers on the numeration 
systems $G^{(1)}, \ldots, G^{(s)}$ is ergodic if and only if the discrete parts of the spectra intersect only at $1$.

As mentioned in \cite{glt},
$$
 \lim_{n \rightarrow \infty} \frac{G^{(j)}_n}{(\beta^{(j)})^n} = b_j,
$$
where the constant $b_j$ can be computed by residue calculus. Using $\sim$ for asymptotic equality (if $n \rightarrow \infty$) we obtain 
for fixed $l \in \mathbb{N}$
\begin{align*}
 \exp \left( 2 \pi i \frac{G^{(j)}_n}{(\beta^{(j)})^l} \right) &\sim \exp \left( 2 \pi i b_j (\beta^{(j)})^{n - l} \right)\\
 &\sim \exp \left( 2 \pi i G^{(j)}_{n - l} \right),
\end{align*}
and thus
\begin{equation*}
 \lim_{n \rightarrow \infty} \exp \left( 2 \pi i \frac{G^{(j)}_n}{(\beta^{(j)})^l} \right) = \lim_{n \rightarrow \infty} \exp \left( 2 \pi i G^{(j)}_{n - l} 
\right) = 1.
\end{equation*}
Now assume that $C_j = \gcd\left(a^{(j)}_0, \ldots, a^{(j)}_{d_j - 1}\right)$, then for every $k \in \mathbb{N}$ 
there exists an integer $n_0$ with $C_j^k \mid G^{(j)}_{n}$ for all $n \geq n_0$ and there exist no integers $C', n_0' \in \mathbb{N}$ with $\gcd(C', C_j) = 1$ 
such 
that $C' \mid G^{(j)}_{n}$ 
for all $n \geq n_0'$. Thus $\Gamma_j$ can be written as
\begin{equation*}
 \Gamma_j = \left\{ \exp \left( 2 \pi i \frac{c}{C_j^m (\beta^{(j)})^l} \right) \colon m,l,c \in \mathbb{N} \cup \{0\} \right\}.
\end{equation*}
By the assumptions of the theorem the spectra only intersect at $1$. Hence the product dynamical system is ergodic. Moreover, the
unique ergodicity of the product system follows since the product dynamical system is isomorphic to a group rotation and therefore purely discrete.

Now it remains to show that $\left(\psi_{\boldsymbol \beta}(n)\right)_{n >0}$ is uniformly distributed in $[0,1)^s$. We follow the ideas given in 
\cite{hit}. Lemma~\ref{lem1} implies that
the extended $\beta$-adic Monna map $\psi_{\beta}$ transports the measure $\mu$ on the system of $G$-expansions to the Lebesgue measure on the unit interval.
Furthermore, by Lemma \ref{lem1} the Monna map $\psi_\beta$, together with its pseudo inverse, defines an isomorphism, which implies that the dynamical system 
on
$[0,1)^s$ defined by the transformation $T_{\boldsymbol \beta} \colon [0,1)^s \rightarrow [0,1)^s$, with $T_{\boldsymbol \beta}= \psi_{\boldsymbol \beta} \tau_{\boldsymbol \beta}
\psi_{\boldsymbol \beta}^{-1}$ is uniquely ergodic, where $\tau_{\boldsymbol \beta}(x)=(\tau_{\beta_1}(x),\dots,\tau_{\beta_s}(x) )$. The proof is complete by applying Birkhoff's ergodic theorem.
\end{proof}

\section{Corner avoidance}\label{Sec:corner}

In this and the following sections we keep the assumptions and notations of the previous sections. In particular, we consider an $s$-dimensional
$\bbeta$-adic Halton sequence $(\psi_{\bbeta}(n))_{n >0}$ subject to $s$ numeration systems $G^{(1)}, \dots, G^{(s)}$ with coefficients $\boldsymbol 
a^{(i)}=\left(a_0^{(i)},\dots,a_{d_i}^{(i)}\right)$ and let us assume that $a_0^{(i)}=\dots =a_{k_i}^{(i)}>a_{k_i+1}^{(i)}$. As before we denote by
$\beta^{(i)}$ the characteristic root of the numeration system $G^{(i)}$ for $1\leq i \leq s$ and for $1\leq j \leq d_i$ we denote by $\beta^{(i)}_j$ their 
conjugates. By convention we write $\beta^{(i)}_1=\beta^{(i)}$. Let $K^{(i)}$ be the Galois closure of 
$\Q\left(\beta^{(i)}\right)$, i.e. $K^{(i)}=\Q\left(\beta^{(i)}_1,\dots,\beta^{(i)}_{d_i}\right)$, and let $\Gamma^{(i)}=\Gal(K^{(i)}/\Q)$ be its Galois 
group.


Let $\mathbf h=(h^{(1)},\dots,h^{(s)})\in\{0,1\}^s$ be a corner of the unit cube $[0,1)^s$. We define the hyperbolic distance from 
$\mathbf x=(x^{(1)},\dots,x^{(s)})\in[0,1)^s$ to the corner $\mathbf h$ by 
$$ \|\mathbf x\|_{\mathbf h}=\prod_{i=0}^s \left|x^{(i)}-h^{(i)}\right|,$$
where $|\cdot|$ is the usual absolute value. For the rest of the paper we aim to prove the following theorem:

\begin{theorem}\label{th:corners}
 Assume that $K^{(i)}\cap K^{(j)}=\Q$ for all $1\leq i <j \leq s$. Then the $\boldsymbol \beta$-adic Halton sequence
 $\left(\psi_{\boldsymbol \beta}(n)\right)_{n >0}$ 
avoids corners, i.e. for any $\varepsilon>0$ there exists a constant $C_{\varepsilon,\boldsymbol \beta}$ such that
\begin{equation}\label{eq:CornerAvoid} \|\psi_{\boldsymbol \beta}(N)\|_{\mathbf h}>\frac{C_ {\varepsilon,\boldsymbol \beta}}{N^{H/2+\varepsilon}}\end{equation}
where 
$$H=\left\{ \begin{array}{ll} 2 & \;\; \text{if}\;\; \mathbf h=(0,\dots,0)\\
s&  \;\; \text{if}\;\; \mathbf h=(1,\dots,1)\\
1+\sum_{i=1}^s h^{(i)} & \;\; \text{in all other cases} \end{array} \right. .$$
\end{theorem}

In order to prove Theorem \ref{th:corners} we use Schmidt's famous Subspace Theorem \cite{Schmidt:1971}. More precisely we use a version of the Subspace 
Theorem that goes back to Schlickewei~\cite{Schlickewei:1977}. To formulate the Theorem properly we 
have to discuss the absolute values of a number field $K/\Q$ with maximal order $\ord$. It is well known that for each  
prime ideal $\p$ in $\ord$ there is a finite field $\FF=\ord/\p$ with $p^{n_\p}$ elements, where $p>0$ is the characteristic of $\FF$. The 
number $n_\p$ is called the local degree of $\p$. We assign to each prime ideal $\p$ of $\ord$ an absolute value $|\alpha|_\p:=p^{n_\p v_\p(\alpha)}$ 
on $K$, where $v_{\p}(\alpha)$ is the unique integer such that $(\alpha)=\p^{v_\p(\alpha)}\mathfrak q$ and $\p$ and $\mathfrak q$ are coprime fractional 
ideals. We call all these 
absolute values the finite or non-Archimedean canonical absolute values.

Similarly we define infinite absolute values. Therefore we note that a number field $K/\Q$ of degree $d=[K:\Q]$ has $d$ embeddings 
$\sigma:K\hookrightarrow\C$. If the embedding $\sigma$ is real we put $n_\sigma=1$ and put $n_\sigma=2$ otherwise. For each embedding $\sigma$ we define 
an absolute value on $K$ by $|\alpha|_\sigma=|\sigma\alpha|^{n_\sigma}$ where $|\cdot|$ is the usual absolute value in $\C$. Note that with a complex embedding 
$\sigma$ also its complex conjugate $\bar\sigma$ is an embedding inducing the same absolute value. We call all these absolute values $|\cdot|_\sigma$ the 
infinite or Archimedean canonical absolute values.

We denote by $M(K)$ the set of all canonical absolute values. The product formula (see e.g. \cite[Chapter III, Theorem 1.3]{Neukirch:ANTE}) states that
\begin{equation}\label{eq:ProductFormula}
 \prod_{\nu \in M(K)} |\alpha|_\nu=1.
\end{equation}
for all algebraic numbers $\alpha\in K$. Note that the product is defined since $|\alpha|_\nu \neq 1$ for at most finitely many absolute values $\nu$.

With these notations we are able to state the absolute Subspace Theorem in its simplest form suitable for our proof (cf. \cite[Chapter V, Theorem 
1D]{Schmidt:1991}).

\begin{theorem}[Subspace Theorem]\label{th:Subspace}
Let $K$ be an algebraic number field with maximal order $\ord$ and let
$S \subset M(K)$ be a finite set of absolute
values which contains all of the Archimedean ones. For each $\nu \in S$
let $L_{\nu,1}, \cdots, L_{\nu,n}$ be $n$ linearly independent linear
forms in $n$ variables with coefficients in $K$. Then for given
$\delta>0$, the solutions of the inequality
$$
\prod_{\nu \in S} \prod_{i=1}^n |L_{\nu ,i}(\mathbf x) |_{\nu} <\overline{|\mathbf x|}^{-\delta}
$$
with $\mathbf x=(x_1,\dots,x_n) \in \ord^n$ and $\mathbf x \not = \mathbf 0$, where
\[\overline{|\mathbf x|}= \max \left\{|x_i|_\nu \: :\: 1 \leq i \leq n, \nu\in M(K), \nu \; \text{Archimedean}\right\},\]
lie in finitely many proper subspaces of $K^n$.
\end{theorem}

\begin{remark}
 We want to note that this version of the Subspace Theorem is not state of the art. More recent versions of the Subspace Theorem due to 
Evertse and Schlickewei \cite{Evertse:2002a} or Evertse and Ferretti \cite{Evertse:2013} exist. These results have the extra feature that they yield 
a bound for the number of subspaces. However, it is not possible, even in principle, to determine these finitely many subspaces. Therefore an effective version 
of Theorem \ref{th:corners} cannot be achieved with the methods of this paper. Since 
an application of such a more recent version of Schmidt's Subspace Theorem is technical and yields no extra gain in view of Theorem~\ref{th:corners} we use the 
cited, older and simpler version.
\end{remark}

\section{Auxiliary Results}\label{Sec:aux}

This section has two aims. First, we want to relate the corner avoidance property to a Diophantine inequality to which we can apply the Subspace Theorem 
(Theorem~\ref{th:Subspace}) and second, we provide several estimates and relations which will be used in the next section.

A first step is to recognize for which $n$ the extended Monna map $\psi_{\beta}$ takes its extremal values. The minimal values are easy to describe but the 
maximal values depend on $\mathbf a$ and in particular on the index $k$ such that $a_0=a_1=\dots=a_k>a_{k+1}$. 

\begin{lemma}\label{lem:extreme}
If
\[\psi_{\beta}(N)\leq \frac{\beta^{d-1}}{\beta^n (\beta^{d-1} + \beta^{d-2} + \ldots + 1)},\]
then $N$ has a regular $G$-expansion of the form 
$$N=\varepsilon_n(N) G_n+ \dots ,$$
i.e. $\varepsilon_i(N)=0$ for all $i=0,1\ldots , n-1$. If
\[1-\psi_{\beta}(N)\leq  \frac{\beta^{d-1}}{\beta^n (\beta^{d-1} + \beta^{d-2} + \ldots + 1)},\]
then $N$ has a regular $G$-expansion $N=\sum_{i} \varepsilon_i(N) G_i$, where
the digits $\varepsilon_i(N)$ with $i<n$ satisfy
$\varepsilon_i(N)=a_0$ if $k+1\nmid i$ or $i=0$ and $\varepsilon_i(N)=a_0-1$ otherwise. 
\end{lemma}

\begin{proof}
As mentioned in \cite[Proof of Lemma 3]{glt}, for a cylinder set $Z$ with $K$ digits, the functions $F_{K,r}$ given in \eqref{mu} follow the same recurrence 
relation as the base
sequence $G_n$. Thus
\begin{equation*}
 (F_{K,j} - a_0 F_{K,j-1} - \ldots - a_{j-1} F_{K,0}) \geq 0, \quad \text{ for } 0 \leq j \leq d,
\end{equation*}
and hence 
\[\mu(Z_j)\geq \frac{\beta^{d-1}}{\beta^j (\beta^{d-1} + \beta^{d-2} + \ldots + 1)},\]
where $Z_j$ is a cylinder set with at most $j$ digits.

Let us assume that
\[\psi_{\beta}(N)\leq \frac{\beta^{d-1}}{\beta^n (\beta^{d-1} + \beta^{d-2} + \ldots + 1)},\]
but $N$ has regular $G$ expansion
$$N=\varepsilon_k(N) G_k+ \dots ,$$
where $\varepsilon_k(N)\neq 0$ and $k<n$. Then we have
$$\psi_{\beta}(N)\geq \mu(Z_k)\geq  \frac{\beta^{d-1}}{\beta^{k} (\beta^{d-1} + \beta^{d-2} + \ldots + 1)}>
\frac{\beta^{d-1}}{\beta^n (\beta^{d-1} + \beta^{d-2} + \ldots + 1)},$$
where $Z_k=Z(\stackrel{k\,\text{times}}{\overbrace{0,\dots,0}})$. Hence the first statement of the lemma is proved.

In order to prove the second statement we claim the following

\begin{claim}\label{claim:max}
 Let $N<G_{n}$ be an integer with regular $G$-expansion
 $$N=\varepsilon_0(N) G_0+ \dots + \varepsilon_{n-1}(N)G_{n-1}$$
 which yields the word $w_N=\varepsilon_0(N)\dots \varepsilon_{n-1}(N)$. The word $w_N$ is maximal with respect to lexicographic order if and only 
 if $\varepsilon_i(N)=a_0$ for all indices $i$ such that $k+1\nmid i$ or $i=0$ and $\varepsilon_i(N)=a_0-1$ for all other indices.
\end{claim}

\begin{proof}[Proof of the claim]
First, let us note that the digit expansion given in the claim is a regular $G$-expansion. We prove this by induction on $n$. Indeed for $n\leq k+1$ we have
$$G_n=a_0+a_0G_1+\dots+a_0G_{n-1}+1>a_0+a_0G_1+\dots+a_0G_{n-1}$$
and inequality \eqref{regular} is satisfied. Now assume that all those expansions are regular for all $m<n$. We have to prove that also the expansion for $n$ 
is regular, i.e. we have to show that inequality \eqref{regular} holds for $n$, too. Therefore assume that $n=n'(k+1)+r$ with $r<k+1$. We obtain
\begin{multline*}
 a_0+\dots+a_0 G_{n'(k+1)-1}+(a_0-1)G_{n'(k+1)}+a_0G_{n'(k+1)+1}+\dots+a_0G_n\\
 <a_0G_{n'(k+1)}+a_0G_{n'(k+1)+1}+\dots+a_0G_n\leq G_{n+1},
\end{multline*} 
where the first inequality holds by the induction hypothesis with $m=n'(k+1)-1$ and the second inequality holds because of the recursion defining the sequence 
$G_n$.

Since we assume that the $G$-expansion is regular no digit can be larger than $a_0$. Therefore our construction for $n\leq k$ is optimal. 
Since
$$a_0+\dots+a_0G_{k+1}>G_{k+2}$$
we deduce that $\varepsilon_{k+1}(N)<a_0$. In general we have
\begin{multline*}
G_{n'(k+1)+1}=a_0 G_{n'(k+1)}+\dots+a_0G_{(n'-1)(k+1)+1}+a_{k+1}G_{(n'-1)(k+1)}+\dots \\
<a_0 G_{n'(k+1)}+\dots+a_0G_{(n'-1)(k+1)+1}+a_0G_{(n'-1)(k+1)},
\end{multline*}
hence all digits $\varepsilon_{(k+1)n'}(N)$ are less than $a_0$. Hence our construction is optimal.
\end{proof}

We continue the proof of Lemma \ref{lem:extreme}. Since $1=\sum \mu(Z)$, where the sum runs over all cylinder sets of length $n$, we deduce that
\begin{equation}\label{eq:max_eq}
1-\psi_{\beta}(N)\leq \frac{\beta^{d-1}}{\beta^n (\beta^{d-1} + \beta^{d-2} + \ldots + 1)}\leq \mu(Z)
\end{equation}
only if $Z$ is the cylinder set of length $n$ such that $\psi_{\beta}(Z)$ lies closest to the right edge. But, by the definition of $\psi_{\beta}$ this is 
exactly the cylinder set $Z$ of the form $Z(\varepsilon_0(N),\dots, \varepsilon_{n-1}(N))$ such that the corresponding word $w_N=\varepsilon_0(N)\dots 
\varepsilon_{n-1}(N)$ 
is maximal with respect to the lexicographic order. Therefore $N$ satisfies inequality~\eqref{eq:max_eq} provided its regular $G$-expansion lies in $Z$ and by 
Claim~\ref{claim:max} the proof is complete.
\end{proof}

Let $G$ be a numeration system as above and $\beta$ its characteristic root and $\beta=\beta_1,\beta_2,\dots,\beta_d$ the conjugates of $\beta$. Then we have
\begin{equation}\label{eq:Rec_explicit}
G_n=\sum_{j=1}^{d} b_j \beta_j^n,
\end{equation}
where $b_j\in K=\Q(\beta_1,\dots,\beta_d)$. Note that by our assumptions the characteristic polynomial is 
irreducible and has therefore no double zero.

\begin{lemma}\label{lem:Cramer}
We have $b_j\in\Q(\beta_j)$ and they are conjugates. In particular, if $\sigma\in\Gamma=\Gal(K/\Q)$ satisfies
$\sigma(\beta_j)=\beta_k$ for some $1\leq j,k \leq d$, then $\sigma(b_j)=b_k$. Moreover, there exists a rational
integer $D_0$ such that all $D_0b_j$ are algebraic integers. In fact $D_0$ can be chosen to be the discriminant of the order 
$\Z[\beta_1,\dots,\beta_{d}]$.
\end{lemma}

\begin{proof}
 We know that $b_j=\frac{\beta_j^d -1}{\beta_j-1} \frac{1}{P'(\beta_j)}$, where $P$ is the characteristic polynomial of $G$
 (e.g. see \cite[page 2 resp. Section 4]{steiner}). That $D_0$ can be chosen to be the discriminant of the order 
$\Z[\beta_1,\dots,\beta_{d}]$ is due to the definition of the different and discriminant (see e.g. \cite[Section III.2]{Neukirch:ANTE})
and some well-known relations between them.
\end{proof}

Let us fix the corner $\mathbf h\in\{0,1\}^s$ and let $I=\{i\: : \: h^{(i)}=0\}$ and $J=\{j\: :\: h^{(j)}=1\}$. 
If $i\in I$ let $n_i$ be such that 
$$
N= \varepsilon_{n_i}^{(i)}(N) G^{(i)}_{n_i}+ \dots+ \varepsilon_{N_i}^{(i)}(N) G^{(i)}_{N_i},
$$
with $\varepsilon_{n_i}^{(i)}(N)\neq 0$. If $i\in J$ let $n_i$ be maximal such that $\varepsilon_{\ell}^{(i)}(N)= a_0^{(i)}$ for $\ell=0$ and all $\ell<n_i$ 
such that $k_i+1\nmid \ell$, $\varepsilon_{\ell}^{(i)}(N)= a_0^{(i)}-1$ for all $1\leq \ell < n_i$ such that $k_i+1| \ell$.
Then Lemma \ref{lem:extreme} yields
\begin{equation}\label{Eq:CornerAvoid}\left\|\psi_{\boldsymbol \beta}(N)\right\|_{\mathbf h}>c_1 \prod_{i=1}^s \left|\beta^{(i)}\right|^{-n_i}\end{equation}
for some effective computable constant $c_1$. For instance one can choose
$$c_1=\prod_{i=1}^s\left(\frac{\left(\beta^{(i)}\right)^{d_i-1}}{\left(\beta^{(i)}\right)^{d_i-1} + \left(\beta^{(i)}\right)^{d_i-2} +\dots + 1}\right).$$
Thus in order to prove Theorem \ref{th:corners} it is enough to prove
$$ \prod_{i=1}^s \left|\beta^{(i)} \right|^{-n_i}>\frac{\tilde C_{\varepsilon,\boldsymbol \beta}}{N^{H/2+\varepsilon}} $$
with $n_i$ as defined above and some constant $\tilde C_{\varepsilon,\boldsymbol \beta}$. Using the explicit formula \eqref{eq:Rec_explicit} for $G^{(i)}_n$ we 
obtain for $i\in I$
\begin{equation}\label{eq:corner_x}
N= \sum_{j=1}^{d_i} x^{(i)}_j \;\; \text{with} \;\; x_j^{(i)}=\left(\beta_j^{(i)}\right)^{n_i}
\stackrel{:=c_j^{(i)}}{\overbrace{\sum_{n=0}^{N_i-n_i} b_j^{(i)}\varepsilon_{n_i+n}^{(i)}(N) \left(\beta_j^{(i)}\right)^n}}\quad \forall i \in I.
\end{equation}
In case of $i\in J$ we want to decompose $N$ into a sum similar to \eqref{eq:corner_x}. Therefore let us compute
\begin{align*}
 N=& \sum_{j=1}^{d_i} \left(\sum_{n=0}^{n_i-1}\varepsilon_n^{(i)}(N) b_j^{(i)}\left(\beta_j^{(i)}\right)^n
 +\sum_{n=n_i}^{N_i} \varepsilon_n^{(i)}(N) b_j^{(i)}\left(\beta_j^{(i)}\right)^n\right)\\
=&  \sum_{j=1}^{d_i} \left(\sum_{n=0}^{n_i-1}\varepsilon_n^{(i)}(N) b_j^{(i)}\left(\beta_j^{(i)}\right)^n
 +\left(\beta_j^{(i)}\right)^{n_i}\stackrel{:= \tilde c_j^{(i)}}
 {\overbrace{\sum_{n=0}^{N_i-n_i} b_j^{(i)}\varepsilon_{n_i+n}^{(i)}(N) 
\left(\beta_j^{(i)}\right)^n}}\right)\\
=& 1+\sum_{j=1}^{d_i}\left( a_0 \sum_{n=0}^{n_i-1} b_j^{(i)}\left(\beta_j^{(i)}\right)^n
-\sum_ {n=0}^{\left\lfloor \frac{n_i-1}{k_i+1}\right\rfloor}b_j^{(i)}\left(\beta_j^{(i)}\right)^{n(k_i+1)}+\tilde c_j^{(i)} 
\left(\beta_j^{(i)}\right)^{n_i}\right)\\
=& 1+\sum_{j=1}^{d_i} \left(b_j^{(i)}\frac{a_0\left(1-\left(\beta_j^{(i)}\right)^{n_i}\right)}{1-\beta_j^{(i)}}-
b_j^{(i)}\frac{1-\left(\beta_j^{(i)}\right)^{\left\lceil \frac{n_i}{k_i+1} \right\rceil(k_i+1)}}
{1-\left(\beta_j^{(i)}\right)^{k_i+1}} +\tilde c_j^{(i)} \left(\beta_j^{(i)}\right)^{n_i}\right)\\
=& 
\stackrel{:=A_i}{\overbrace{1+\sum_{j=1}^{d_i}\frac{a_0b_j^{(i)}}{1-\beta_j^{(i)}}-\sum_{j=1}^{d_i}\frac{b_j^{(i)}}{1-\left(\beta_j^{(i)}\right)^{k_i+1}}}}
\\
&+ \sum_{j=1}^{d_i}\left(\beta_j^{(i)}\right)^{n_i}\stackrel{:=c_j^{(i)}}{\overbrace{\left(\tilde c_j^{(i)}- \frac{a_0b_j^{(i)}}{1-\beta_j^{(i)}}+ 
\frac{b_j^{(i)}\left(\beta_j^{(i)}\right)^{\left\lceil \frac{n_i}{k_i+1} 
\right\rceil(k_i+1)-n_i}}{1-\left(\beta_j^{(i)}\right)^{k_i+1}}\right)}}\\
=& A_i+\sum_{j=1}^{d_i}\left(\beta_j^{(i)}\right)^{n_i} c_j^{(i)}
\end{align*}
Note that 
$$0\leq \left\lceil \frac{n_i}{k_i+1} \right\rceil(k_i+1)-n_i<k_i+1.$$
Therefore we obtain for $i\in J$ the equation $N-A_i= \sum_{j=1}^{d_i} x^{(i)}_j$,
where $x^{(i)}_j=\left(\beta_j^{(i)}\right)^{n_i} c_j^{(i)}$. In order to unify this equation with equation \eqref{eq:corner_x} we put $A_i=0$ 
for $i\in I$. Moreover, we define $\mathcal A=\{A_i\: :\: 1\leq i\leq s\}$ and obtain
\begin{equation}\label{eq:corner_y}
N-A_i= \sum_{j=1}^{d_i} x^{(i)}_j
\end{equation}
for all $1\leq i \leq s$. Further we define 
$$I_a:=\{i \: :\: A_i=a\} \quad \forall a\in\mathcal A,$$
which yields a partition of the set $\{1,2,\dots,s\}$ and we put $i_a=\min I_a$ for all $a\in \mathcal A$. Finally let us note that without 
loss of generality we may assume that $0\in\mathcal A$. Since otherwise we replace $N$ by $N-a$ for some $a\in\mathcal A$. By such a change of $N$ we 
change only the quantities $A_i$ and $c_j^{(i)}$ for $1\leq i \leq s$, $1\leq j\leq d_j$. However the proof of Theorem \ref{th:corners} only depends on the 
size of these 
quantities but not on their specific form.

\begin{remark}
Let us emphasize here that inequality \eqref{eq:CornerAvoid} in Theorem \ref{th:corners} can be replaced by 
\begin{equation}\label{ieq:strong_corner}
\|\psi_{\boldsymbol \beta}(N)\|_{\mathbf h}>\frac{C_ {\varepsilon,\boldsymbol \beta}}{N^{\max\{2,|\mathcal A|\}/2+\varepsilon}}.
\end{equation}
Indeed we will prove this inequality which implies inequality \eqref{eq:CornerAvoid} since obviously we have $H\geq \max\{2,|\mathcal A|\}$.
\end{remark}

In our considerations below it will be important to estimate the quantities $A_i$ for $1\leq i \leq s$ and $c_j^{(i)}$ for $1\leq i \leq s$, $2\leq j \leq 
d_i$. Moreover let $D$ be the common denominator of all $a\in \mathcal A$ and all $c_j^{(i)}$. Thus $D$ is the 
smallest positive, rational integer $D$ such that $Dc_j^{(i)}$ and $Da$ are algebraic integers for all $1\leq i\leq s$, $1\leq j\leq d_i$ and
$a\in\mathcal A$. Also an estimate of the quantity $D$ will be needed in the course of the proof of Theorem \ref{th:corners}. Therefore we prove the following 
lemma:

\begin{lemma}\label{lem:c_est}
 With the notations above we have $\left|c_j^{(i)}\right|,|A_i|,|D|\leq C$, provided $j\neq 1$, where $C$ depends only on the $s$-tuple 
of numeration systems $(G^{(1)},\dots, G^{(s)})$. Moreover, we have $A_i\in\Q$.
\end{lemma}

\begin{proof}
 First, note that $c_j^{(i)}$ in case of $i\in I$ and $\tilde c_j^{(i)}$ in case of $i \in J$ can be estimated for all $2\leq j \leq d_i$ by a geometric 
series, hence both are bounded by 
a constant neither depending on $N$ nor $n_i$. Moreover, for fixed $i$ the common denominator is the integer $D_0$ from Lemma \ref{lem:Cramer}. But also the 
other components of $c_j^{(i)}$ in case of $i\in J$ only depend on $\beta^{(i)}_j$ and the constants $a_j^{(i)}$. Hence $\left|c_j^{(i)}\right|\leq C$ for 
all $1\leq i \leq s$, $2\leq j\leq d_i$ and a sufficiently large constant $C$. Also the common denominator of all $c_j^{(i)}$ is bounded by a constant 
neither depending on $N$ nor on $n_i$. Similarly we obtain $|A_i|\leq C$ and the denominators of all $A_i$ are again bounded by a 
constant. Hence, also $|D|\leq C$ for a sufficiently large constant $C$.

Therefore we are left to show that $A_i$ is 
rational for all $1\leq i \leq s$. But $A_i$ can be written as the trace of some algebraic number $\alpha \in K^{(i)}$ and is therefore rational. Indeed we 
have $A_i=0$ for all $i\in I_0$ and for all other indices $i$ we have
\begin{multline*}
A_i=1+\sum_{j=1}^{d_i}\frac{a_0b_j^{(i)}}{1-\beta_j^{(i)}}-\sum_{j=1}^{d_i}\frac{b_j^{(i)}}{1-(\beta_j^{(i)})^{k_i+1}}\\
=\mathrm{Tr}_{\Q(\beta^{(i)})/\Q}\left(\frac{1}{d_i}+\frac{a_0b_1^{(i)}}{1-\beta^{(i)}}-\frac{b_1^{(i)}}{1-(\beta^{(i)})^{k_i+1}} \right)\in\Q.
\end{multline*}
\end{proof}

\section{Application of the Subspace Theorem}\label{Sec:Subspace}

A key step proving Theorem \ref{th:corners} is the following Proposition:

\begin{proposition}\label{prop:2corner}
Let $a,b\in \mathcal A$ not necessarily distinct and $\varepsilon>0$. Then there exists a constant $C_{\varepsilon}$ such that
\begin{equation}\label{eq:2corner_estimate}
C_{\varepsilon} N^{1+\varepsilon}> \prod_{i\in I_a \cup I_b} \left| \beta_1^{(i)} \right|^{n_i}.
\end{equation}
\end{proposition}

Once we have established \eqref{eq:2corner_estimate} it is easy to deduce inequality \eqref{ieq:strong_corner} and hence Theorem \ref{th:corners}.
The case $|\mathcal A|\leq 2$ is a direct consequence of Proposition \ref{prop:2corner} and inequality~\eqref{Eq:CornerAvoid}. Therefore let us assume
that $|\mathcal A|>2$. Then according to Proposition \ref{prop:2corner} there exists a constant $C=C_{2\varepsilon/|\mathcal A|}$ such that
\begin{equation*}
\begin{split}
\left(C N^{1+2\varepsilon/|\mathcal A|}\right)^{\frac{|\mathcal A|(|\mathcal A|-1)}2}>&\prod_{a,b\in\mathcal A\atop a<b} \prod_{i\in I_a\cup 
I_b} \left| \beta_1^{(i)} \right|^{n_i}=\prod_{i=1}^s \left| \beta_1^{(i)} \right|^{(|\mathcal A|-1)n_i}\\
\geq & \frac{c_1^{|\mathcal A|-1}}{\left\|\psi_{\boldsymbol \beta}(N)\right\|_{\mathbf h}^{|\mathcal A|-1}}
\end{split}
\end{equation*}
which immediately yields inequality \eqref{ieq:strong_corner}.

Therefore our task is to prove Proposition \ref{prop:2corner}, which takes the rest of the paper. Without loss of generality we may assume that $b=0$ in   
Proposition \ref{prop:2corner} and we fix $a\in\mathcal A$. Moreover, we fix the number field $K$ which is the compositum of the number fields 
$K^{(i)}$ with $i\in I_0 \cup I_a$. Further we denote by $\Gamma$ the Galois group of $K$. Since we assume that all the fields $K^{(i)}$ are linearly 
disjoint we have $\Gamma=\prod_{i \in I_0 \cup I_a} \Gamma^{(i)}$.

In order to apply the Subspace Theorem, we have to determine which valuations are of 
particular interest in view of Theorem \ref{th:corners}.
Thus let 
\[S^{(i)}_j=\left\{\nu\in M(K)\: :\: \nu=|\cdot|_\p, \,\p\left|\left(\beta_j^{(i)}\right)\right.\right\}\]
and
\[\Sfin=\bigcup_{i\in I_0\cup I_a} \bigcup_{j=1}^{d_i} S^{(i)}_j\]
In particular, $\Sfin$ is the set of finite places of $M(K)$ that lie above all the primes that divide some principal ideal $\left(\beta_j^{(i)}\right)$ 
with $i\in  I_0\cup I_a$ and $1\leq j\leq d_i$. Moreover let
\[\Sinf=\left\{\nu\in M(K)\: :\: \nu=|\cdot|_\sigma, \, \sigma\in \Gamma\right\}\]
be the set of all canonical Archimedean absolute values in $K$. In view of Theorem~\ref{th:Subspace} we write $S=\Sfin\cup \Sinf$.

Our next step is to construct linear forms $L_{\nu,j}^{(i)}$ for all $\nu\in S$. The aim is to find linear forms that match the conditions of the Subspace 
Theorem on the one hand and yield small values when they are evaluated at the point $\mathbf x=\left(x_j^{(i)}\right)$ on the other hand. Before we  
construct these linear forms we want to note that for $i\in I_a$ we have
$$ x_1^{(i_a)}+\dots+x_{d_{i_a}}^{(i_a)}=N-a=x_1^{(i)}+\dots+x_{d_{i}}^{(i)}.$$
Thus we can express $x_1^{(i)}$ as a linear combination of $x_j^{(i_a)}$ for $1\leq j \leq d_{i_a}$ and $x_j^{(i)}$ for $2\leq j \leq 
d_{i}$. Loosely speaking we can eliminate the unknown $x_1^{(i)}$ provided $i\neq i_a$ and $i\in I_a$. Therefore
$$X=\left\{X^{(i)}_j \: : \: i\in I_0\cup I_a,\;\;  2\leq j \leq d_i\right\} \cup \left\{X^{(i_0)}_1, X^{(i_a)}_1 \right\}$$
is the set of indeterminates over which we construct our linear forms $L_{\nu,j}^{(i)}\in K[X]$.

First, we consider the linear forms for non-Archimedean places.
Therefore let us assume that $\nu\in S^{(\ell)}=\bigcup_{m=1}^{d_\ell} S_m^{(\ell)}$. Then we define:
\begin{itemize}
 \item If $(i,j)=(i_a,1)$ and $\ell\in I_{a}\setminus\{i_a\}$, then
 $$L_{\nu,1}^{(i_a)}(X)=\left(X_1^{(i_a)}+\dots+X_{d_{i_a}}^{(i_a)}\right)-\left(X_2^{(\ell)}+\dots+X_{d_{\ell}}^{(\ell)}\right).$$
 \item and in all other cases
 $$L_{\nu,j}^{(i)}(X)=X_j^{(i)}.$$
\end{itemize}
Note that for a fixed valuation $\nu\in S^{(\ell)}_m$ the linear forms $L_{\nu,j}^{(i)}(X)$ are linearly independent over $K$. Let us remind that by the 
product formula \eqref{eq:ProductFormula} we have
\[\prod_{\nu \in M(K) \atop \nu \; \text{non-Archimedean}} |\alpha|_\nu=\frac{1}{\left|N_{K/\Q}(\alpha)\right|}.\]




Now, let us turn to the Archimedean absolute values. As explained above every $\nu\in \Sinf$ corresponds to some $\sigma \in\Gamma$. Let us emphasize
again that $\sigma$ is unique up to complex conjugation, and every $\sigma\in \Gamma$ determines a $\nu \in \Sinf$. Let us assume for the moment that $\sigma$ 
is a complex embedding. Then we also write 
$L_{\nu,j}^{(i)}=L_{\sigma,j}^{(i)}$ for some $\sigma\in \Gamma$. Note that $L_{\sigma,j}^{(i)}=L_{\bar \sigma,j}^{(i)}$ for all $\sigma \in \Gamma$. Moreover, 
let us write $\|\cdot\|_\sigma=|\cdot|_\nu^{1/2}$ thus we have $|\cdot|_\nu=\|\cdot\|_\sigma \|\cdot\|_{\bar \sigma}$. Altogether we have with these notations
$$\prod_{\nu \in \Sinf} \left|L_{\nu,j}^{(i)}(\mathbf x)\right|_\nu=\prod_{\sigma \in \Gamma} \left\|L_{\sigma,j}^{(i)}(\mathbf x)\right\|_\sigma.$$
Now we have enough notations to define our linear forms. In case that $\sigma\left(\beta_j^{(i)}\right)\neq \beta_1^{(i)}$ we define
\[L_{\sigma,j}^{(i)}(X)=X^{(i)}_j.\]
In case that $\sigma(\beta_j^{(i)})= \beta_1^{(i)}$ and
\begin{itemize}
\item if $i\in I_a$ and $i\neq i_a$, then we put
\[L_{\sigma,j}^{(i)}(X)=\left(X^{(i_a)}_1+\cdots+X^{(i_a)}_{d_{i_a}}\right)-\left(X^{(i)}_2+\cdots+X^{(i)}_{d_{i}}\right),\]
\item if $i=i_a$ and $a\neq 0$, then we put
\[L_{\sigma,j}^{(i_a)}(X)=\left(X^{(i_0)}_1+\cdots+X^{(i_0)}_{d_{i_0}}\right)-\left(X^{(i_a)}_1+\cdots+X^{(i_a)}_{d_{i_a}}\right),\]
\item
and finally if $i=i_0$, then we put $L_{\sigma,j}^{(i_0)}(X)=X^{(i_0)}_{j}$.
\end{itemize}
Since by assumption $\beta_1^{(i)}$ is Pisot and therefore real, $\sigma\left(\beta_j^{(i)}\right)=\beta_1^{(i)}$ if and only if 
$\bar\sigma\left(\beta_j^{(i)}\right)=\beta_1^{(i)}$. Therefore the linear forms $L_{\nu,j}^{(i)}$ are well defined.

Our next task is to compute
$$\prod_{j} \prod_{\nu\in S} \left|L_{\nu,j}^{(i)}(\mathbf x)\right|_\nu$$
for fixed $i\in I_0\cup I_a$. Let us consider the case $i\neq i_0, i_a$, first. In this case we have $j>1$. 
Note that in case that  $\sigma\left(\beta^{(i)}_j\right)=\beta_1^{(i)}$ and $i\in I_a\setminus\{i_a\}$ we have
$$L_{\sigma,j}^{(i)}(\mathbf x)=
\stackrel{=N-a}{\overbrace{\left(x^{(i_a)}_1+\cdots+x^{(i_a)}_{d_{i_a}}\right)}}-
\stackrel{=N-a-x_1^{(i)}}{\overbrace{\left(x^{(i)}_2+\cdots+x^{(i)}_{d_{i}}\right)}}
=x_1^{(i)}.$$
With this remark we obtain
$$
\prod_{j=2}^{d_i} \prod_{\nu \in S} \left|L_{\nu,j}^{(i)}(\mathbf x)\right|_\nu=\prod_{j=2}^{d_i} \prod_{\nu \in S} \left|x_j^{(i)}\right|_\nu \times
\prod_{j=2}^{d_i} \prod_{\sigma \in \Gamma \atop \sigma\left(\beta_j^{(i)}\right)=\beta_1^{(i)}} 
\frac{\left\|x_1^{(i)}\right\|_\sigma}{\left\|x_j^{(i)}\right\|_\sigma}
$$
The first product on the right side can be estimated by the product formula, i.e. we have 
$$\prod_{\nu \in S} \left|x_j^{(i)}\right|_\nu = \prod_{\nu \in S} \left|c_j^{(i)}\right|_\nu \times \prod_{\nu \in S} \left|\beta_j^{(i)}\right|_\nu =
\prod_{\nu \in S} \left|c_j^{(i)}\right|_\nu.$$
For the second product note that
$$
\prod_{j=2}^{d_i} \prod_{\sigma \in \Gamma \atop \sigma\left(\beta_j^{(i)}\right)=\beta_1^{(i)}} 
\frac{\left\|x_1^{(i)}\right\|_\sigma}{\left\|x_j^{(i)}\right\|_\sigma}= \prod_{j=1}^{d_i} \prod_{\sigma \in \Gamma \atop \sigma\left(\beta_j^{(i)}\right)=\beta_1^{(i)}} 
\frac{\left\|x_1^{(i)}\right\|_\sigma}{\left\|x_j^{(i)}\right\|_\sigma}=\prod_{\sigma \in \Gamma} 
\frac{\left\|x_1^{(i)}\right\|_\sigma}{\left|x_1^{(i)}\right|}=
\frac{\left|N_{K/\Q}\left(x_1^{(i)}\right)\right|}{\left|x_1^{(i)}\right|^{|\Gamma|}}.
$$
Therefore we obtain
\begin{equation}\label{eq:ProdEst1}
\begin{split}
\prod_{j=2}^{d_i} \prod_{\nu \in S}& \left|L_{\nu,j}^{(i)}(\mathbf x)\right|_\nu= 
 \prod_{j=2}^{d_i} \prod_{\nu \in S} \left|c_j^{(i)}\right|_\nu \times \frac{\left|N_{K/\Q}\left(x_1^{(i)}\right)\right|}{\left|x_1^{(i)}\right|^{|\Gamma|}}\\
= & \mathcal O\left(\frac{\left|c_1^{(i)}\right|^{|\Gamma|}\left|N_{K/\Q}\left(\left(\beta_1^{(i)}\right)^{n_i}\right)\right|}{\left|x_1^{(i)}\right|^{|\Gamma|}} 
\right)= \mathcal O\left(\frac{\left|N_{K/\Q}\left(\left(\beta_1^{(i)}\right)^{n_i}\right)\right|}{\left|\left(\beta_1^{(i)}\right)^{n_i}\right|^{|\Gamma|}} \right).
\end{split}
\end{equation}

Next, let us consider the case $i=i_a$ but $a\neq 0$. Let $\sigma\in \Gamma$ be such that 
$\sigma\left(\beta^{(i_a)}_j\right)=\beta_1^{(i_a)}$. Then we use the following fact:
$$L_{\sigma,j}^{(i_a)}(\mathbf x)=
\stackrel{=N-a}{\overbrace{\left(x^{(i_a)}_1+\cdots+x^{(i_a)}_{d_{i_a}}\right)}}-
\stackrel{=N}{\overbrace{\left(x^{(i_0)}_1+\cdots+x^{(i_0)}_{d_{i_0}}\right)}}=a.$$ 
Thus we get
\begin{equation*}
 \begin{split}
\prod_{j=1}^{d_{i_a}} & \prod_{\nu \in S} \left|L_{\nu,j}^{(i_a)}(\mathbf x)\right|_\nu\\
=& \prod_{j=1}^{d_{i_a}} \prod_{\nu \in S} \left|x_j^{(i_a)}\right|_\nu \times 
\prod_{j=1}^{d_{i_a}} \prod_{\sigma \in \Gamma \atop \sigma\left(\beta_j^{(i_a)}\right)=\beta_1^{(i_a)}} \frac{\|a\|_\sigma}{\left\|x_j^{(i_a)}\right\|_\sigma}
\times \prod_{\ell\in I_a \setminus \{i_a\}} \prod_{\nu \in S_1^{(\ell)}} \frac{\left|x_1^{(\ell)}\right|_\nu}{\left|x_1^{(i_a)}\right|_\nu}.
\end{split}
\end{equation*}
Here the first two products are similarly treated as in the case above. To estimate the third product let us note that for all $\ell\in I_a \setminus \{i_a\}$ and
all $\nu \in S_1^{(\ell)}$ we have that $\left|\left(\beta_1^{(i_a)}\right)^{n_{i_a}}\right|_\nu, \left|c_1^{(i_a)}\right|_\nu\geq 1$. Moreover,
we have $\left|c_1^{(\ell)}\right|_\nu=\mathcal O(1)$ provided that $\nu\in S_1^{(\ell)}$ and therefore
$$ \prod_{\nu \in S_1^{(\ell)}}\left|x_1^{(\ell)}\right|_\nu= \mathcal O \left(\prod_{\nu \in S_1^{(\ell)}}\left|\left(\beta_1^{(\ell)}\right)^{n_{i_a}}\right|\right)
= \mathcal O \left(\frac{1}{\left|N_{K/\Q}\left(\left(\beta_1^{(\ell)}\right)^{n_\ell}\right)\right|} \right).$$
Altogether this yields
\begin{equation}\label{eq:ProdEstia}
 \begin{split}
\prod_{j=1}^{d_{i_a}} & \prod_{\nu \in S} \left|L_{\nu,j}^{(i_a)}(\mathbf x)\right|_\nu\\
=& \prod_{j=1}^{d_{i_a}} \prod_{\nu \in S} \left|c_j^{(i_a)}\right|_\nu \times\frac{\left|N_{K/\Q}(a)\right|}{\left|x_1^{(i_a)}\right|^{|\Gamma|}} \times 
\prod_{\ell\in I_a \setminus \{i_a\}}\prod_{\nu \in S_1^{(\ell)}} \frac{1}{\left|\left(\beta_1^{(i_a)}\right)^{n_{i_a}}\right|_\nu}\cdot  
\frac{\left|x_1^{(\ell)}\right|_\nu}{\left|c_1^{(i_a)}\right|_\nu}\\
=& \mathcal O\left(\frac{1}{\left|\left(\beta_1^{(i_a)}\right)^{n_{i_a}}\right|^{|\Gamma|}\prod_{\ell\in I_a \setminus \{i_a\}} 
\left|N_{K/\Q}\left(\left(\beta_1^{(\ell)}\right)^{n_\ell}\right)\right| } \right).
\end{split}
\end{equation}

Finally let us consider the case $i=i_0$ and similarly as before we obtain the inequality
\begin{equation}\label{eq:ProdEsti0}
 \begin{split}
\prod_{j=1}^{d_{i_0}} & \prod_{\nu \in S} \left|L_{\nu,j}^{(i_0)}(\mathbf x)\right|_\nu= \prod_{j=1}^{d_{i_0}} \prod_{\nu \in S} \left|x_j^{(i_0)}\right|_\nu 
\times \prod_{\ell\in I_0 \setminus \{i_0\}} \prod_{\nu \in S_1^{(\ell)}} \frac{\left|x_1^{(\ell)}\right|_\nu}{\left|x_1^{(i_0)}\right|_\nu}\\
=& \mathcal O\left(\frac{1}{\prod_{\ell\in I_0 \setminus \{i_0\}} \left|N_{K/\Q}\left(\left(\beta_1^{(\ell)}\right)^{n_\ell}\right)\right| } \right)
\end{split}
\end{equation}

Let us combine the estimates \eqref{eq:ProdEst1}, \eqref{eq:ProdEstia} and \eqref{eq:ProdEsti0} and note that 
$c_1^{(i)}\left(\beta_1^{(i)}\right)^{n_i}=N+\mathcal O(1)$. Then we obtain
\begin{equation}\label{eq:Subspace_est}
\prod_{\nu \in S} \prod_{i,j} \left|L_{\nu, j}^{(i)}(\mathbf x) \right|_{\nu}=\mathcal O\left(\frac{N^{|\Gamma|}}{\prod_{i\in I_0 \cup I_a} 
\left(\left|\beta_1^{(i)}\right|^{n_i}\right)^{|\Gamma|}}\right)
\end{equation}

In order to apply the Subspace Theorem we have to compute $\overline{|\mathbf x|}$. Since $N-A_i=x_1^{(i)}+\dots+x_{d_i}^{(i)}$ and $x_j^{(i)}=\mathcal O(1)$ 
unless $j=1$ by Lemma \ref{lem:c_est}, we deduce that $x_1^{(i)}= N+\mathcal O(1)$ for all $i\in I_0 \cup I_a$. By Lemma~\ref{lem:Cramer} all the $x_j^{(i)}$ 
are 
conjugated, hence $\overline{|\mathbf x|}=N+\mathcal O(1).$ 

Now applying the Subspace Theorem we obtain 
$$\frac{\tilde C_{\varepsilon,\mathbf \beta}N^{|\Gamma|}}{\prod_{i\in I_0 \cup I_a} (|\beta_1^{(i)}|^{n_i})^{|\Gamma|}}>N^{-\varepsilon|\Gamma|}$$
for some constant $\tilde C_\varepsilon=C_\varepsilon^{|\Gamma|}$. Therefore either Proposition \ref{prop:2corner} holds or all solutions $\mathbf x$ subject to 
\eqref{eq:corner_y} lie on finitely many proper subspaces $T\subset K^{|X|}$. 
But, we can show that these subspaces contain only bounded solutions and in particular there exist at most finitely many exceptional $N$. In particular, the 
proof of Proposition \ref{prop:2corner} is complete by proving the following lemma:

\begin{lemma}\label{lem:finite_sol}
All solutions $\mathbf x$ subject to \eqref{eq:corner_y} that lie on a proper subspace $T\subset K^{|X|}$ are bounded, i.e. $\overline{|\mathbf x|}<C$ for 
some constant $C$.
\end{lemma}

\begin{proof}

In order to prove Lemma \ref{lem:finite_sol} it is sufficient to prove that only bounded solutions $\mathbf x$ satisfying  
\eqref{eq:corner_y} lie on a proper linear subspace $T\subset K^{|X|}$, i.e. satisfy a homogenous equation of the form
\begin{equation}\label{eq:ProperSubSpace}
 t_1^{(i_0)} x_1^{(i_0)}+ t_1^{(i_a)} x_1^{(i_a)}+\sum_{i\in I_0\cup I_a} \sum_{j=2}^{d_i} t_j^{(i)} x_j^{(i)}=0,
\end{equation}
where $\left|x_j^{(i)}\right|=\mathcal O(1)$ unless $j=1$ and not all $t_j^{(i)}$ vanish. In case that $a=0$ we immediately obtain 
$\left|x_1^{(i_0)}\right|=\mathcal O(1)$,
hence $N$ is bounded and therefore also $\overline{|\mathbf x|}$. 

In case that $a\neq 0$ we consider the linear system
\begin{equation}\label{eq:LinSys}
\begin{split}
 0=& t_1^{(i_0)} x_1^{(i_0)}+ t_1^{(i_a)} x_1^{(i_a)}+\stackrel{=\mathcal O(1)}{\overbrace{\sum_{i\in I_0\cup I_a} \sum_{j=2}^{d_i} t_j^{(i)} x_j^{(i)}}}\\
 a=& x_1^{(i_0)}-x_1^{(i_a)} +\stackrel{=\mathcal O(1)}{\overbrace{\sum_{j=2}^{d_i} x_j^{(i_0)}-x_j^{(i_a)}}} .
\end{split} 
\end{equation}
Let us assume for the moment that $t_1^{(i_a)}\neq -t_1^{(i_0)}$. If we add $t_1^{(i_a)}$ times the second equation to the first equation in the linear
system~\eqref{eq:LinSys}, then we see that $x_1^{(i_0)}$ can be written as the linear combination of the $x_j^{(i)}$ with $j>1$, i.e. $x_1^{(i_0)}$ is bounded 
and $N$ and $\overline{|\mathbf x|}$ are bounded too.

Therefore we may assume that $t_1^{(i_a)}= -t_1^{(i_0)}$ and the linear system \eqref{eq:LinSys} turns into
\begin{equation}\label{eq:Reduced}\sum_{i\in I_0\cup I_a}\sum_{j=2}^{d_i} u_j^{(i)} x_j^{(i)}=b\end{equation}
where $u_j^{(i_0)}=t_j^{(i_0)}-t_1^{(i_0)}$, $u_j^{(i_a)}=t_j^{(i_a)}-t_1^{(i_a)}$, $u_j^{(i)}=t_j^{(i)}$ if $i\neq i_0,i_a$ and $b=at^{(i_a)}$.
Let us recall that the Galois group of $K$ is of the form 
$\Gamma=\prod_{i\in I_0 \cup I_a} \Gamma^{(i)}$. Let us take some $\sigma\in \Gamma$ such that 
$\sigma\left(\beta_m^{(\ell)}\right)=\beta_1^{(\ell)}$ and $\sigma|_{K^{(i)}}=\mathrm{id}$ for all $i\neq \ell$. If we apply $\sigma$ to equation 
\eqref{eq:Reduced} we obtain
\begin{multline*}
 \sigma\left(u_m^{(\ell)}\right)\sigma\left(x_m^{(\ell)}\right)+\sum_{i\in I_0\cup I_a} \sum_{2\leq j \leq d_i \atop (i,j)\neq(\ell,m)}
\sigma\left(u_j^{(i)}\right) \sigma\left(x_j^{(i)}\right)=\\
\sigma\left(u_m^{(\ell)}\right) x_1^{(\ell)}+\stackrel{\mathcal O(1)}{\overbrace{\sum_{i \in I_0\cup I_a}\sum_{j=2}^{d_i} \tilde u_j^{(i)} x_j^{(i)}}}= 
\sigma(b),
\end{multline*}
where $\tilde u_j^{(i)}=\sigma\left(u_k^{(i)}\right)$ and $k$ satisfies $\sigma\left(\beta_k^{(i)}\right)=\beta_j^{(i)}$.
Therefore either $\sigma\left(u_m^{(\ell)}\right)=0$, i.e. $u_m^{(\ell)}=0$, or $\left|x_1^{(\ell)}\right|=\mathcal O(1)$. Therefore varying $m$ and $\ell$ 
over all possibilities we obtain either $u_i^{(j)}=0$ for all admissible pairs $(i,j)$ or the 
solution $\mathbf x$ is bounded. Therefore we may assume that all $t_j^{(i_0)}=t_1^{(i_0)}$, $t_j^{(i_a)}=t_1^{(i_a)}=-t_1^{(i_0)}$ and $t_j^{(i)}=0$ if $i\neq i_0,i_a$.
Furthermore, we may assume that $t_1^{(i_0)}\neq 0$ since otherwise all $t_j^{(i)}$ would vanish which contradicts the fact that $T$ is a proper subspace.
Therefore the linear system \eqref{eq:LinSys} yields
$$ 0=x_1^{(i_0)}-x_1^{(i_a)} +\sum_{j=2}^{d_i} x_j^{(i_0)}-x_j^{(i_a)}=a\neq 0,$$
a contradiction.
\end{proof}

\begin{remark}
We want to emphasize that the only point where we used the assumption that the $K^{(i)}$ are linearly disjoint over $\Q$ is in the proof of Lemma 
\ref{lem:finite_sol}. If we do not assume that the $K^{(i)}$ are linearly disjoint, then we cannot assure that certain coefficients $t_j^{(i)}$ in 
\eqref{eq:ProperSubSpace} vanish and the proof breaks down.
\end{remark}

\begin{remark}
In view of Lemma \ref{lem:finite_sol} the constant $C_{\varepsilon,\boldsymbol \beta}$ in Theorem \ref{th:corners} depends on the coefficients $t_j^{(i)}$ 
from the equation defining 
the subspace $T$. Since the ineffectiveness of the absolute Subspace Theorem the constant $C_{\varepsilon,\boldsymbol \beta}$ remains ineffective.
\end{remark}

\section*{Acknowledgment}

The second author was supported by the Austrian Science Fund (FWF) under the project P~24801-N26. Moreover, we want to thank the anonymous referee for his 
helpful comments, which considerably improved Section \ref{Sec:Subspace}.


\end{document}